\documentclass[a4paper,10pt]{amsart}

\usepackage{mathrsfs}
\usepackage{amssymb,amsmath, amsfonts, amsthm}
\usepackage{latexsym}
\usepackage{color} %{\color{red}..............OK}
\usepackage[dvips]{epsfig}
\usepackage{graphicx}

\newtheorem{defi}{Definition}[section]
\newtheorem{theorem}[defi]{Theorem}
\newtheorem{lemma}[defi]{Lemma}
\newtheorem{corollary}[defi]{Corollary}
\newtheorem{proposition}[defi]{Proposition}
\newtheorem{remark}[defi]{Remark}

\def\N{\mathbb{N}}
\def\R{\mathbb{R}}
\def\B{\mathbb{B}}

\title{Sign-changing radial solutions for the Schr\"odinger-Poisson-Slater problem.}
\author{Isabella Ianni}
\address{Department of Mathematics, SUN - Seconda Universita di Napoli\\ viale Lincoln 5, Caserta, Italy}
\email{isabella.ianni@unina2.it}
\begin{document}

\maketitle

\begin{abstract}
We consider the Schr\"odinger-Poisson-Slater (SPS) system in $\R^3$  and  a nonlocal SPS type equation in balls of $\mathbb R^3$ with Dirichlet boundary conditions.
We show that for every $k\in\mathbb N$ each problem considered admits a nodal radially symmetric solution
which  changes sign exacly $k$ times in the radial variable.

Moreover when the domain is the ball of $\mathbb R^3$ we obtain the existence of radial global solutions for the associated nonlocal parabolic problem  having $k+1$ nodal regions at every time. 
\end{abstract}

\section{Introduction}
We consider the Schr\"odinger-Poisson-Slater (SPS) problem in $\mathbb R^3$
\begin{equation}\label{sistemaSPSInRN}\left\{\begin{array}{lr}-\Delta u+u+\phi u -|u|^{q-1}u=0 \quad\mbox{ in }\mathbb R^3\\-\Delta\phi=u^2\qquad\qquad\quad\qquad\qquad\mbox{ in }\mathbb R^3\\\lim_{|x|\rightarrow +\infty}\phi(x)=0\end{array}\right.
\end{equation}

From a physical point of view systems  like \eqref{sistemaSPSInRN} appear is
semiconductor theory to model the evolution of an electron ensemble
in a semiconductor crystal (see \cite{Markovich, Mauser, 
BokanowsiLopezSoler, SanchezSoler}).
In this context the Poisson
potential $\phi$ comes from the repulsive interactions among
electrons while the nonlinear term
$|u|^{p-1}u$ is introduced as a correction to the
repulsive Poisson potential to explain different phenomena observed
from experimentations (for instance in simulations with
superlattices structures). In particular in the case $p=\frac{5}{3}$
this term is usually known as \emph{Slater correction}, since it
comes from a term that was first introduced by Slater (1951) as a
local approximation for the exchange term in the Hartree-Fock
equations (see \cite{Slater}, \cite{Dirac}). For this reasons we refer to system \eqref{sistemaSPSInRN}
as Schr\"odinger-Poisson-Slater system (SPS).

A different justification of system \eqref{sistemaSPSInRN}
 can be found also in \cite{BenciFortunato} where
it is proposed as a model formally describing the interaction of a
charged particle with its own electrostatic field. 

%%%%%%%%%%%%%%%%%%%%%%%%%%%%%%%%%%%%%%%%%%%%%%%%%%%%%%%%%%%%%%%%%%%%%%%%%%%%%%%%%%%%%%%%%%%%%%%%%%%%%%%%%%

$\;$\\ 
From a mathematical point of view system \eqref{sistemaSPSInRN} shows several difficulties, the 
nonlinear nature of it being due both  to the power-type nonlinearity in the first equation and to the
coupling, and has been object of many investigation in the last years (we recall among others the papers \cite{AmbrRuizMult, AzzPom, BenciFortunato, coclite, DapMugSol, DapMugNonEx, davenia, kikuchi, RuizJFA}, see also \cite{IanRui}).
\\

As shown by recent results 
%(\cite{AmbrRuizMult, AzzPom, DapMugSol, DapMugNonEx, davenia, kikuchi, RuizJFA},see also \cite{IanRui})
the structure of the solution set of \eqref{sistemaSPSInRN} depends strongly on the value of $q$ of the power-type nonlinearity. 

For $q\leq 2$ and $q\geq 5$ system \eqref{sistemaSPSInRN} doesn't admit any nontrivial solution (see \cite{DapMugNonEx, kikuchi, RuizJFA}), while when $q\in (2,5)$ existence and multiplicity results have been proved using variational techniques.

%We briefly recall that he action functional associated with the system is strongly indefinite. To overcome this problem one usually %reduces to a scalar equation with  both a power-type and a nonlocal nonlinearity, and hence studies the existence of critical points for %the one variable functional that doesn't present such a strongly indefinite nature (for this reduction method see for instance %\cite{BenciFortunato}).

Precisely  in  \cite{DapMugSol,  kikuchi, RuizJFA} the existence of at least one nontrivial radial solution is proved while in  \cite{AmbrRuizMult} they show the existence of infinitely many radial solutions.

For completeness we recall that mostly these existence results are obtained through min-max procedures, and one needs to restrict the energy functional to the natural constrained of the radial functions to overcome the problem of the lack of compactness of the Sobolev embeddings in the unbounded domain $\mathbb R^3.$
We also point out that in \cite{AmbrRuizMult} no information about the sign of the solutions is given.
On the other hand in \cite{AzzPom} the existence of a positive ground state solution has been proved but it is still an open problem whether it is radial or not.
$$$$
In the present paper we analyze more deeply the structure of the radial bound states set for problem \eqref{sistemaSPSInRN}.

We show the existence of infinitely many radially symmetric sign-changing solutions which are distinguished by the number of nodal regions, more precisely we prove the existence of radial solutions which have a prescribed number of nodal domains.

Our proof combines a dynamical approach together with a limit procedure and it is mainly inspired by \cite{WeiWeth}. Up to our knowledge this is the first time that such a different approach (not variational) is used in the context of the Schr\"odinger-Poisson-Slater problems.\\

Our main result is the following
 
\begin{theorem}\label{teoremaEsistenzaInRN}
%Let $q\in (3, 5)$. 
Let $q\in [3, 5)$. For every integer $k\geq 2$, \eqref{sistemaSPSInRN} admits a  couple of  radial solutions $(\pm u,\phi)$ such that $\pm u$ changes sign precisely $k-1$ times in the radial variable.
%
%\textcolor{red}{TOGLI QUESTO:} 
%
%and such that\textcolor{red}{......ora $E$ e' quello del sistema (vedi azzolliniPomponio)o dell'eq. equivalente...??.....} $E(u)\leq M_k$. 
\end{theorem}
$$$$
In order to obtain this result we will study first the existence of sign changing radial solutions for the following semilinear elliptic equation 
with Dirichlet boundary condition \begin{equation}\label{problemaDiDirichletNellaPalla}\left\{\begin{array}{lr}-\Delta u+u+u\int_{B_R}\frac{u^2(y)}{|x-y|}dy-|u|^{q-1}u=0 \quad\mbox{ in }\B_R\\u=0 \quad\mbox{ on }\partial \B_R\end{array}\right.
\end{equation}
where $\B_R$ is the ball of radius $R$ in $\R^3.$ This problem, with both  local and nonlocal nonlinearities, has been also investigated in \cite{Ruiz} when $q\in (1,2)$.

We recall  that  solutions of \eqref{problemaDiDirichletNellaPalla} are critical points of the energy functional $E:H^1_{0}(\B_R)\rightarrow \mathbb R$
given by
$$E(u)=\frac{1}{2}\int|\nabla u(x) |^2dx+\frac{1}{2}\int u(x)^2dx+\frac{1}{4}\int\int\frac{u^2(x)u^2(y)}{|x-y|} dxdy-\frac{1}{q+1}\int |u|^{q+1}(x) dx$$

$$$$

Our second main result is then the analogus of Theorem \ref{teoremaEsistenzaInRN} for problem \eqref{problemaDiDirichletNellaPalla}:
\begin{theorem}\label{teoremaEsistenzaNellaPalla}
Let $q\in [3, 5)$. For every $R\geq 1$ and every integer $k\geq 2$, \eqref{problemaDiDirichletNellaPalla} admits a couple of  radial solutions $\pm u$ changing sign precisely $k-1$ times in the radial variable. Moreover %if $q>3$ then 
there exists a constant $C_k>0,$ independend of $R,$ such that $E(\pm u)\leq C_k$. 
\end{theorem}
$$$$
The proof of theorem \ref{teoremaEsistenzaNellaPalla} relies on a dynamical method. We find solution of the elliptic problem \eqref{problemaDiDirichletNellaPalla} looking for equilibria in the $\omega$-limit set of trajectories of the associated parabolic problem, namely of

\begin{equation}\label{parabolicProblem}\left\{\begin{array}{lr}u_t-\Delta u+u+u\int_{\B_R}\frac{u^2(y)}{|x-y|}dy-|u|^{q-1}u=0 \quad\mbox{ in }\B_R\times [0,\infty)\\u=0 \quad\mbox{ on }\partial \B_R\times [0,\infty)\\
u(\cdot,0)=u_0\mbox{ on }\B_R\end{array}\right..\end{equation} 
$$$$
This nonlocal  initial boundary value parabolic  problem has been studied in  \cite{IanniPreprintParabolic} by the author himself.
$$$$

Solutions of elliptic equations via the corresponding parabolic flow  can be found in the literature (we recall among others  \cite{ CazenaveLions, Quittner, ContiMerizziTerracini}),
in particular here we follow an approach introduced in \cite{WeiWeth} in the contest of symmetric systems of two coupled Schr\"odinger equations but which can be applied,with proper modifications, also to scalar equations with odd nonlinearities, like our problem \eqref{problemaDiDirichletNellaPalla}.

This method consists in selecting special initial data on the boundary of the domain of attraction of an asymptotically stable equilibrium in order to obtain equilibria with a fixed number of changes of sign.

It relies on a crucial monotonicity property of the semilinear parabolic problem \eqref{parabolicProblem} stating that the number of zeros is not increasing along the flow %(Lemma \ref{NumRegioniNonCresceLungoTraiettoria})
and combines the study of the parabolic flow with a topological argument based on the use of the Krasnonelskii genus.

  We underline that, for fixed $k$, the solutions found in Theorem \ref{teoremaEsistenzaNellaPalla} satisfy an energy bound independent on the radius $R$ of the domain. This  is the starting point to prove Theorem \ref{teoremaEsistenzaInRN}, throught a limit procedure on the radius of the domain.\\

Moreover, as a byproduct in the proof of Theorem \ref{teoremaEsistenzaNellaPalla} we obtain also the following result related to the existence of sign-changing global solutions for the parabolic problem  \eqref{parabolicProblem} that we believe to be of independent interest
\begin{theorem}\label{teoParabolico}
Let $q\in [3,5).$ For any integer $k\geq 2$ there exists a couple  $\pm u_k:\mathbb B_R\times[0,\infty)\rightarrow \mathbb R$ of global radial solutions of 
\eqref{parabolicProblem} such that, for all $t\geq 0$ $\pm u_k(\cdot, t)$ has exactly $k-1$ changes of sign in the radial variable.
\end{theorem}

$$$$
We point out that all the results obtained in the present paper are  obtained for values $q\in [3,5)$, the case  $q\in (2,3)$ being still opened.  

This is due to the fact that first the existence of solutions for the parabolic problem \eqref{parabolicProblem} when $q\in (2,3)$ is still open (see \cite{IanniPreprintParabolic}), second the geomeric properties of the energy functional depend strongly on the value of $q$ (see the proof of Proposition \ref{UniformHBound} where we need to restrict to the case $q\geq 3$).
We recall also that it is still an open problem whether the (PS) property holds or not in the gap $q\in (2,3)$, precisely it has not yet been proved the existence of a bounded Palais-Smale sequence when $q$ belongs to this gap.

Our difficulties seem to be strictly related to the ones one find when dealing with this open problem. 
$$$$

We briefly describe the paper's organization. 

In Section \ref{Preliminari} we collect some notations and preliminaries.
%on the variational framework for \eqref{} 
 In Section \ref{sectionParabolic} we recall the properties of the nonlocal parabolic problem \eqref{parabolicProblem} (local and global existence results, regularity, compactness properties) which have been studied in \cite{IanniPreprintParabolic} by the author himself.
Moreover we prove the monotonicity of the number of zeros along the parabolic flow.

In Section \ref{SPazioFinitoDim} we define a family of finite dimensional spaces $(W_k)_{k\in\mathbb N}$ such that the restriction of the energy functional on it is unbounded from below and bounded from above uniformly on the radius $R$ of the domain.
The use of $W_k$ will be crucial for the proof of Theorem \ref{teoremaEsistenzaNellaPalla}, in particular to obtain uniform estimates of the energy.
 
Section \ref{sectionDimoTeoremiPalla} is therefore devoted to the proofs of  Theorems \ref{teoremaEsistenzaNellaPalla} and \ref{teoParabolico}.

Section \ref{sectionTeoInR3} contains the proof of Theorem \ref{teoremaEsistenzaInRN} which is obtained throught a limit procedure on the radius of the balls $\B_R.$ In particular we controll the number of zeros while passing to the limit using ODE techniques,  Strauss Lemma and maximum principles.\\\\\\

\section{Notations and preliminaries}\label{Preliminari}
Let us fix some notations. 

$\B_R:=\left\{x\in\R^3:|x|< R\right\}$ is the unit ball of $\mathbb R^3$ of radius $R$.

For an open $\Omega\subseteq \mathbb R^3,$ 
$(L^r(\Omega),\|\cdot\|_{L^3(\Omega)})$ is the usual Lebesgue space, we may also write the norm simply as $\|\cdot\|_r$ when there is no misunderstanding about the integration set.

$(W^{s,r}(\Omega),\|\cdot\|_{W^{s,r}(\Omega)})$ and $W^{s,r}_0(\Omega)$ are the usual Sobolev or Sobolev-Slobodeckii spaces
we may also write the norm simply as $\|\cdot\|_{s,r}$ when there is no misunderstanding about the integration set.

In particular we  write $H^1(\Omega)$ resp. $H^1_0(\Omega)$ instead of $W^{1,2}(\Omega)$ resp. $W^{1,2}_0(\Omega)$ and in this case we may denote the norm symply with
$\|u\|:=\|u\|_{H^{1}_0(\Omega)}=\int_{\Omega} (|\nabla u |^2+|u|^2)dx.$ 
 
$C^{m,\alpha}(\Omega)$ is the  subspaces of $C^{m}(\Omega)$ consisting of functions whose $m$-th order partial derivatives are {\it locally H\"older continuous with exponent $\alpha$ in $\Omega$}.

If $\Omega$ is bounded then we denote by $(C^{m,\alpha}(\bar\Omega), \|u\|_{C^{k,\alpha}(\bar\Omega)})$ the Banach space of all the functions belonging to $C^{m}(\bar\Omega)$ whose $m$-th order partial derivatives are {\it uniformly H\"older continuous with exponent $\alpha$ in $\bar\Omega$}.
 endowed with the usual norm.
 
$D^{1,2}(\mathbb R^3)$ is the closure of $C^{\infty}_0(\mathbb R^3)$ with respect to the norm $\|u\|_{D^{1,2}}=\|\nabla u\|_{L^2(\mathbb R^3)}.$
$$$$

The following facts are known (see for instance \cite[Lemma 0.3.1]{IanniPhdThesis}) and \cite[Lemma 2.1]{RuizJFA}):
\begin{lemma}\label{preliminariSuPhi}
For any $v\in H^1(\mathbb R^3)$ let $$\phi_v(x):=\int_{\mathbb R^3}\frac{v^2(y)}{|x-y|}dy.$$
Then 
\begin{itemize}
\item[i)] $\phi_v \in D^{1,2}(\mathbb R^3)$ and  there exists $C>0$ (independent of $v$) such that
$$\|\phi_v\|_{D^{1,2}}\leq C\|v\|^2.$$
Hence in particular there exists $C>0$ (independent of $v$) such that
\begin{equation}\label{HLSsemplice}\int\!\!\int\frac{w^2(x)v^2(y)}{|x-y|}dxdx\leq C \|w\|^2\|v\|^2\qquad\forall w\in H^1(\mathbb R^3).	\end{equation}
\item[ii)] $\phi_v$ is the unique weak solution in $D^{1,2}(\mathbb R^3)$ of the equation $-\Delta\phi_v=v^2$ in $\mathbb R^3.$\\
\item[iii)] If $v$ is radial, then $\phi_v$ is radial and has the following expression
$$\phi_v(r)=\frac{1}{r}\int_0^{+\infty} v^2(s)s\min\{r,s\}ds.$$
\item[iv)] Let $v_n, v\in  H^1(\mathbb R^3),$ radial and  satisfying $v_n\rightharpoonup v$ in $H^1(\mathbb R^3)$, then $\phi_{v_n}\rightarrow \phi_v$ in $D^{1,2}(\mathbb R^3).$
\end{itemize}
\end{lemma}

Here and in the following for $v\in H^1_0(\B_R),$ we write again $v$ also for the trivial extension of $v$ in the whole $\mathbb R^3$, which belongs to $H^1(\R^3),$ one has that the nonlocal term which appears in the equation \eqref{problemaDiDirichletNellaPalla} coincides with the restriction of $\phi_v$ to $\B_R.$
%$$\int_{\B_R}\frac{v^2(y)}{|x-y|}dy=\phi_v(x),\ x\in \B_R$$.\textcolor{red}{prolungando a zero fuori di $\mathbb B_R$ cosi sono definite in tutto $\mathbb R^3$}\\
\\

Hence if $u\in H^1_0(\B_R)$ is a solution of \eqref{problemaDiDirichletNellaPalla} then the couple $(u,\phi_u|_{\B_R})$ is a solution of the 
 SPS system in the ball
$$\left\{\begin{array}{lr}-\Delta u+u+\phi u -|u|^{q-1}u=0 \quad\mbox{ in }\B_R\\-\Delta\phi=u^2\qquad\qquad\quad\qquad\qquad\mbox{ in }\B_R\\u=0\qquad\qquad\qquad\qquad\qquad\quad\mbox{ on }\partial \B_R\end{array}\right.
$$
We underline that $\phi_u$ is not $0$ on $\partial\B_R$.

\begin{remark}
\label{Schauder regularity}
 If $v\in W^{1,p}_0(\B_R)$ $,p>3$ then $\phi_v\in C^{2,\alpha}(\bar\B_R)$ and satisfies $-\Delta\phi_v=v^2$ in $\B_R.$

Indeed by Sobolev embedding ($p>3$) $v\in W^{1,p}_0(\B_R)\hookrightarrow C^{0,\alpha}(\bar\B_R)$ and $v=0$ on $\partial\B_R$. Let $\tilde v$ be the trivial extension of $v$ in $\B_{2R}$ then $\tilde v^2\in C^{0,\alpha}(\bar\B_{2R})$ and so (see \cite[Lemma 4.2 and 4.4]{GilbargTrudinger}) $\phi_v\in C^{2,\alpha}(\bar\B_R)$ and satisfies $-\Delta\phi_v=v^2$ in $\B_R.$
\end{remark}
$\;$\\\\

\section{The associated parabolic problem}\label{sectionParabolic}
We consider the semilinear nonlocal parabolic initial-boundary-value  problem (IBVP) associated with \eqref{problemaDiDirichletNellaPalla} which has been studied in \cite{IanniPreprintParabolic}:
\begin{equation}\label{IBVP}
 \left\{
\begin{array}{l}
 \frac{\partial u}{\partial t}-\Delta u +u=F(u)\qquad \mbox{ in } \B_R\times (0,+\infty),\\
u=0 \qquad\qquad\qquad\qquad\mbox{ in } \partial\B_R\times (0,+\infty),\\
u(\cdot,0)=u_0 \qquad\qquad\qquad\mbox{ on }\B_R,
\end{array}
\right.
\end{equation}
where the nonlinearity is 
$$F(u):=|u|^{q-1}u-u\int_{\B_R}\frac{1}{|x-y|}u^2(y)dy, \qquad q\in[3,5].$$
$$$$

Next we fix $3< p<\infty$ and $\sigma\in (1+\frac{3}{p}, 2)$ and we consider the function spaces

$$X=\left\{u\in W^{\sigma,p}_{0}(\B_R): u\mbox{ radially summetric}\right\},$$
$$Y=\{u\in C^1(\bar{\B}_R): u \mbox{ is radial and } u=0 \mbox{ on }\partial\B_R\}. $$
$$$$

We have the embedding  $X\hookrightarrow Y,$ since $\sigma>1+\frac{3}{p}.$
$$$$
We recall  the following result related to local existence and regularity which can be found in  \cite{IanniPreprintParabolic}. %\textcolor{red}{\texttt{specifica bene chi e' X}}\\$\;$

\begin{theorem}
%\textcolor{red}{[in $X=W^{\sigma,p}_{0,r}(\B_R)$ for every fixed $\sigma\in (\frac{3}{p}, 2)$ (so in particular we need to require %$p>\frac{3}{2}$]}\\\\
For every $u_0\in X$ the IBVP \eqref{IBVP} has a unique (mild) solution $u(t)=\varphi(t,u_0)\in C([0,T),X)$ with maximal existence time $T:=T(u_0)>0$ which is a classical solution for $t\in (0,T).$

The set $\mathcal G:=\{(t,u_0): t\in [0,T(u_0))\}$ is open in $[0,\infty)\times X,$ and $\varphi:\mathcal G\rightarrow X$ is a semiflow on $X.$ Moreover $\varphi (\cdot,u_0):(0,T(u_0))\rightarrow X$ is  continuous uniformly with respect to  $u_0\in X$ and $\varphi (t,\cdot):X\rightarrow X$ is  locally Lipschitz continuous uniformly with respect to  $t\in (0,T(u_0))$.
%$\varphi\in C^{0,1-}(\mathcal G, X)$
\end{theorem}

%The last statement together with the embedding $X\hookrightarrow Y=\{u\in C^1(\bar{\B}_R): u \mbox{ is radial and } u=0 \mbox{ on }\partial\B_R\}$ 
%\\\\\textcolor{red}{if we also ask $\sigma>1+\frac{3}{p}$ (hence we need also $p>3$)}\\\\ implies 
In particular the following continuity property with respect to the initial datum holds

\begin{corollary}\label{continuitaDatoIn}
%\textcolor{red}{[in $X=W^{\sigma,p}_{0,r}(\B_R)$ for every fixed $\sigma\in (1+\frac{3}{p}, 2)$ (so in particular we need to require $p>3$]}\\\\
For every $u_0\in X$ and every $t\in (0,T(u_0))$ there is a neighborhood $U\subset X$ of $u_0$ in $X$ such that $T(u)>t$ for $u\in U,$ and $\varphi (t,\cdot):(U,\|\cdot\|_X)\rightarrow (Y,\|\cdot\|_{Y})$ is a continuous map.
\end{corollary}

In the following we will often write $\varphi^t(u)$ instead of $\varphi(t,u).$\\

The energy $E$ is strictly decreasing along nonconstant trajectories $t\mapsto \varphi^t(u_0)$ in $X.$
Infact for a classical solution of \eqref{IBVP} we have
\begin{eqnarray}\label{stimaDerivE}
\dot E
&=&\frac{d}{dt}E(u)=
\frac{d}{dt}\int\left(\frac{1}{2}|\nabla u|^2+\frac{1}{2}|u|^2-\frac{1}{q+1}|u|^{q+1}+\frac{1}{4}\phi_uu^2 \right)dx\nonumber\\
&=& \int\left(\nabla u\nabla u_t +uu_t-|u|^{q-1}uu_t+\phi_uuu_t\right)dx\nonumber\\
&=& \int\left(-\Delta u+u-|u|^{q-1}u+\phi_uu\right)u_tdx\nonumber\\
&=& -\int u_t^2 dx=-\|u_t\|_2^2,
\end{eqnarray} 
this property is crucial in order to prove the following global existence result and compactness property (see \cite{IanniPreprintParabolic} for the proof)

\begin{theorem}
%\textcolor{red}{[ qui basta anche $X=W^{\sigma,p}_{0,r}(\B_R)$ for every fixed $\sigma\in (\frac{3}{p}, 2)$]}\\\\
Let $u_0\in X$ and $T=T(u_0)$ be such that the function $t\mapsto E(\varphi^t(u_0))$ is bounded from below in $(0,T.)$ Then $T=\infty$ and for every $\delta>0$ the set $\{\varphi^t(u_0):t\geq\delta\}$ is bounded in $W^{s_1,p}(\B_R)$ for every $s_1\in [\sigma, 2)$ and hence relatively compact in $C^1(\bar{\B}_R).$
\end{theorem}

\begin{corollary}\label{globalExResult}
Let $u_0\in X$ and $T=T(u_0)$ be such that the function $t\mapsto E(\varphi^t(u_0))$ is bounded from below in $(0,T.)$ Then $T=\infty$ and the $\omega$-limit set
$$\omega (u_0)=\cap_{t>0}Clos_{Y}\left(\{\varphi^s(u_0)\ :\ s\geq t\}\right)$$
is a nonempty compact subset of $Y$ consisting of radial solutions of \eqref{problemaDiDirichletNellaPalla}.
\end{corollary}

$\;$\\\\
Next we show that the number of nodal regions of a solution of \eqref{IBVP} is nonincreasing along the flow.\\
Given $u\in X$ we define the number of sign changes in the radial variable $i(u)$ of $u$ 
as the maximal $k\in\N\cup\{0,\infty\}$ such that there exist points $x_1,.., x_{k+1}\in\B$ with $0\leq |x_1|<..<|x_{k+1}|<R$ and $u(x_i)u(x_{i+1})<0$ for $i=1,..,k.$

\begin{lemma}\label{NumRegioniNonCresceLungoTraiettoria}
 Let $u_0\in X$ and $T=T(u_0).$ Then $t\mapsto i(\varphi^t(u_0))$ is nonincreasing in $t\in [0,T)$.
\end{lemma}

\begin{proof}
In view of the semiflow properties, it sufficies to show the inequality
$i(\varphi^{\tau}(u_0))\leq i(u_0)$ for a fixed $\tau\in (0,T).$
Since
$u(t):=\varphi^t(u_0)$  satisfies the equation
$$u_t-\Delta u+ f(x,t)u=0\mbox{ in }\B_R\times (0,\tau]$$ where $f(x,t):= 1+\phi_u(x,t)-|u(x,t)|^{q-1},$ to prove the result we can follow
the arguments in \cite[Theorem 2.1]{ChenPolacik} (see also  \cite[Lemma 2.5]{WeiWeth}% or \cite[Theorem 4]{Sattinger}
). The only thing which remains to be proved is therefore that
$f$ is bounded in $\B_R\times [0,\tau].$

On this scope we observe that $|u|^{q-1}$ is continuous and hence bounded in $\bar\B_R\times [0,\tau]$ (indeed $X\hookrightarrow C^{0,\alpha}_r,$  $u(\cdot, t)\in X$ for every $t$ and $t\mapsto u(t,x)\in C([0,T), X )$).\\
We show now that  $\phi_u(x,t)=\int_{\B_R}\frac{u^2(y,t)}{|x-y|}dy$ is continuous and hence bounded in $\bar\B_R\times [0,\tau].$
The continuity of the function $x\mapsto \phi_u(x,t)$ in $\bar\B_R$  for fixed $t\in[0,\tau]$ is trivial (see Remark \ref{Schauder regularity}).
We fix now  $x\in\bar\B_R$ and show that the function $t\mapsto\phi_u(x,t)$ is continuous in $[0,\tau].$ 

Let $(t_n)_n\subset [0,\tau]$ $t_n\rightarrow_n t_0,$ we want to prove that $\phi_u(x,t_n)\rightarrow_n\phi_u(x,t_0).$ This follows from the following:
\begin{eqnarray*}
\lim_n |\phi_u(x,t_n)-\phi_u(x,t_0)|&=&\lim_n \left|\int_{\B_R}\left(\frac{u^2(y,t_n)}{|x-y|}-\frac{u^2(y,t_0)}{|x-y|}\right)dy\right|\\
%&\leq& \lim_n \int_{\B_R}\frac{1}{|x-y|}\left|u^2(y,t_n)-u^2(y,t_0)\right|dy\\
&\leq& \lim_n \max_{y\in\bar\B}\left|u^2(y,t_n)-u^2(y,t_0)\right|\int_{\B_R}\frac{1}{|x-y|}dy=0
\end{eqnarray*}
\end{proof}
$$$$

\section{A finite dimensional subspace of $X$}\label{SPazioFinitoDim}
In this section we define, for any fixed integer $k\geq 2$  a conveninet $k$-dimensional subspace $W_k\subset X$ such that the restriction to it of the energy functional $E:H^1_0(\B_R)\rightarrow \mathbb R$

$$E(u)=\frac{1}{2}\int|\nabla u(x) |^2dx+\frac{1}{2}\int u(x)^2dx+\frac{1}{4}\int\int\frac{u^2(x)u^2(y)}{|x-y|} dxdy-\frac{1}{q+1}\int |u|^{q+1}(x) dx$$

has some ``good properties'' (it is unbounded below ``at infinity'' and bounded  from above uniformly in the radius $R$ of the domain).

As we will see, we need to pay particular attention to the case $q=3,$ since the geometric properties of the energy $E$ are different (see condition \eqref{quguale3} below).

The spaces $W_k$ will be useful in next section to prove our results.\\\\

First for $0<a<b$ we define the annulus $\mathbb A_{a,b}:=\{x\in\mathbb R^3: a<|x|<b\}.$\\

Then we fix $k$ radial  functions $w_i\in C^2(\mathbb R^3)$ $i=1,..,k$ with disjoint supports which satisfy the following properties:

$$\left\{\begin{array}{lr}w_1>0\quad\mbox{ in }\mathbb B_{\frac{1}{k}}\\
w_1=0\quad \mbox{ in }\mathbb R^3\setminus\mathbb B_{\frac{1}{k}}\end{array}\right.$$

$$\left\{\begin{array}{lr}w_i>0\quad\mbox{ in }\mathbb A_{\frac{i-1}{k},\frac{i}{k}}\qquad i=2,..,k\\w_i=0\quad \mbox{ in }\mathbb R^3\setminus  \mathbb A_{\frac{i-1}{k},\frac{i}{k}}\qquad i=2,..,k\end{array}\right.$$

We can always assume that 
\begin{equation}\label{Nehari}\|w_i\|^2=\|w_i\|^{q+1}_{q+1}\qquad i=1,..,k\end{equation}
(if not we just rescale $w_i$).

We also define
$$M_k:=\max_{i=1,..,k}\|w_i\|^2>0,\quad m_k:=\min_{i=1,..,k}\|w_i\|^2>0.$$
Last for $q=3$ we require also that
\begin{equation}\label{quguale3}
M_k<\frac{1}{k}\end{equation}

%In next two lemmas, related respectively to the case $q>3$ and $q=3,$ the definition and the properties of $W_k$ are stated. 

\begin{lemma}\label{lemmaWminore} Let $$W_k:=\left\{w=w_{(t_1,..,t_k)}:=\sum_{j=1}^k t_jw_j\quad :\  (t_1,..,t_k)\in\R^k   l \right\}\ \subset\  X.$$
Then there exists $C_k>0$ such that 
\begin{equation}\label{stimaE2}E(w)\leq C_k\quad\mbox{ for all }w\in W_k.\end{equation}
Moreover \begin{equation}\label{unboundedBelow2}\lim_{\small{\begin{array}{cr}\|w\|\rightarrow \infty\\w\in W_k
\end{array} }}E(w)=-\infty\end{equation}
\end{lemma}
\begin{proof}
Let $w\in W_k,$ then $w=\sum_{j=1}^k t_jw_j$ where $t_j\in\mathbb R,$ $j=1,..,k,$ hence using the inequality \eqref{HLSsemplice}
\begin{eqnarray}\label{stima_g_k} E(w)&=&\sum_{j=1}^k \left[\frac{t_j^2}{2}\|w_j\|^2-\frac{|t_j|^{q+1}}{q+1}\|w_j\|_{q+1}^{q+1}\right] +\frac{1}{4}\sum_{i,j=1}^k t_j^2t_i^2\int_{\B}\int_{\B}\frac{w_j^2(x)w_i^2(y)}{|x-y|} dxdy\nonumber\\
&&\leq \sum_{j=1}^k \left[\frac{t_j^2}{2}\| w_j\|^2-\frac{|t_j|^{q+1}}{q+1}\|w_j\|^{2}\right] +\frac{1}{4}\sum_{i,j=1}^k t_j^2t_i^2\| w_j\|^2\| w_i\|^2\qquad (\|w_j\|_{q+1}^{q+1}=\|w_j\|^{2})\nonumber\\
&&\leq \sum_{j=1}^k \left[\frac{t_j^2}{2}\| w_j\|^2-\frac{|t_j|^{q+1}}{q+1}\|w_j\|^{2}\right] +\frac{k}{4}\sum_{j=1}^k t_j^4\| w_j\|^4\nonumber\\
&& = \sum_{j=1}^k \| w_j\|^2\left[\frac{t_j^2}{2} +\frac{k}{4} t_j^4\| w_j\|^2-\frac{|t_j|^{q+1}}{q+1}\right]\nonumber\\
&&\leq  \sum_{j=1}^k \| w_j\|^2\left[\frac{t_j^2}{2} +\frac{kM_k}{4} t_j^4-\frac{|t_j|^{q+1}}{q+1}\right]
\end{eqnarray}
Let now be $q>3$, then putting $0<G_k:=\max_{s\in[0,+\infty)}g_{k}(s),$ where $g_{k}(s):=\left[\frac{s^2}{2} +\frac{kM_k}{4} s^4-\frac{|s|^{q+1}}{q+1}\right],$ from \eqref{stima_g_k} it follows that 
$$E(w)\leq  \sum_{j=1}^k \| w_j\|^2G_k\leq kM_kG_k.$$
Moreover, again from \eqref{stima_g_k} we have that
$$\lim_{\|w\|\rightarrow +\infty}E(w)=-\infty,$$
since
$+\infty\leftarrow\|w\|=\sum_{j=1}^k|t_j|\|w_j\|$
iff  there exists (at least one) $ J\in\{1,..,k\}$ s.t. $|t_{J}|\rightarrow +\infty.$ 
\\
If $q=3$ then $g_{k}(s)=\frac{s^2}{2} -s^4\frac{1-k
M_k}{4}$ where $1-kM_k>0$ by our choice in \eqref{quguale3} and we proced in a similar way.
\end{proof}

\begin{remark}
Each function $w_j$ $j=1,..,k$ has support contained in $\mathbb B_1,$ hence in $\mathbb B_R$ for any $R\geq 1.$ 
Therefore the choice of the space $W_k$ as well as the results in the lemma above are independent on the radius $R$ of the domain of problem \eqref{problemaDiDirichletNellaPalla}.
\end{remark}
$$$$

%\newpage
%\section{Energia limitata}
%\textcolor{blue}{IDEA....DA SISTEMARE.....}
%\\\\
%\textcolor{blue}{parto dal dato iniziale  $w$ al tempo $t=0$ e grazie alle stime della sezione precedente so che 
%$$E(w)\leq M$$
%poiche $E$ decresce nel tempo, ho pertanto
%$$1)\qquad E(u)\leq M\quad\forall t\geq 0, (u(t,x)=u(t,x,w))$$
%e anche all'omega limite....e quindi anche per $u$ soluzione\\
%}
%\textcolor{red}{NB: $M$ \`e lo stesso bound  di $E$ della sezione precedente....monotono crescente in  $c_k$ quando $q<3$ ma dipendente anche da $E_S(w_J)$ quando $q=3$....}

%%%%%%%%%%%%%%%%%%%%%%%%%%%%%%%%%%%%%%%%%%%%%%%%%%%%%%%%%%%%%%%%%%%%%%%%%%%%%%%%%%%%%%%%%%%%%%%%%%%%%%%%%%%%%%%%%%%%%%%%%%%%%%%%%%%%%%%%%%%%%%%%%%%%%

\section{Proof of Theorem \ref{teoremaEsistenzaNellaPalla} and Theorem \ref{teoParabolico}} \label{sectionDimoTeoremiPalla}
The proof relies on a dynamical method:
we find solutions of the elliptic problem \eqref{problemaDiDirichletNellaPalla} looking for equilibria in the $\omega$-limit sets of trajectories of the autonomus parabolic problem \eqref{parabolicProblem}. 
In order to obtain equilibria with a fixed number of changes of sign we need to select in a proper way the initial condition. 
This is done following an approach first introduced in \cite{WeiWeth} in the contest of symmetric systems of two coupled Schr\"odinger equations.

This method consists in selecting special initial data on the boundary of the domain of attraction of an asymptotically stable equilibrium.

It relies on the crucial monotonicity property for the number of zeros along the flow (Lemma \ref{NumRegioniNonCresceLungoTraiettoria})
and combines the study of the parabolic flow with a topological argument based on the use of the Krasnonelskii genus. 
For completness we will repeat here the main arguments of \cite{WeiWeth} adapted to our scalar case.
\\

Since the energy functional $E$ is strictly decreasing along nonconstant trajectories (see \eqref{stimaDerivE}) and  $0$ is a strict local minimum for it, it follows that 
the constant solution $u\equiv 0$ is asymptotically stable in $X.$

Let $\mathcal A_*$ be its domain of attraction
 $$\mathcal A_*:=\{u\in X: T(u)=+\infty\mbox{ and }\varphi^t\rightarrow 0 \mbox{ in } X\mbox{ as }t\rightarrow +\infty\}.$$
The asymptotic stability of $0$, the semiflow properties of solutions of \eqref{IBVP} and  the continuous dependence of solutions on intial data (Corollary \ref{continuitaDatoIn}) imply that the set $\mathcal A_*$ is a relatively open neighborhood of $0$ in $X.$\\\\
As in \cite{WeiWeth} we denote with $\partial\mathcal A_*$  the relative boundary of the set $\mathcal A^*$ in $X.$
\\\\
Since $\mathcal A_*$ is open and $0$ is asymptotically stable, the continuous dependence of the semiflow $\varphi$ on the initial values implies that $\partial\mathcal A_*$ is positively invariant under $\varphi.$ 
\\
Moreover $E(u)\geq 0$ for every $u\in \mathcal A_*$ since $E$ is decreasing along trajectories, and hence, by continuity, this is true also for every $u\in\partial\mathcal A_*.$
 As a consequence by Corollary \ref{globalExResult} one has that the solution is global for every initial value $u\in \partial\mathcal A_*,$  the $\omega$-limit  set is nonempty and $\omega(u)\subset\partial\mathcal A_*.$
\\
\\
Since the $\omega$-limit consists of radial solutions of the elliptic problem \eqref{problemaDiDirichletNellaPalla},our aim is to select suitable initial conditions $u$ on $\partial\mathcal A_*$ in a way that any element in $\omega(u)$ has $k-1$ changes of sign. Following \cite{WeiWeth} we define therefore
the closed subset of $X$
$$\mathcal A_k:=\{u\in \partial\mathcal A_*:i(u)\leq k-1\}.$$
Lemma \ref{NumRegioniNonCresceLungoTraiettoria} and the positive invariance of $\partial\mathcal A_*$  for the flow $\varphi$ imply that $\mathcal A_k$ is a positively invariant set for the flow $\varphi.$
\\
Using a topological argument similar to the one in \cite{WeiWeth}, we prove hence the existence of a certain $\bar u\in \mathcal A_k\setminus \mathcal A_{k-1}$ such that $\omega(\bar u)\subset \mathcal A_k\setminus \mathcal A_{k-1}, $ for every $k\geq 2.$
\\
\\
On this scope let's observe that the parabolic problem \eqref{IBVP} has an odd nonlinearity hence the semiflow $\varphi^t$ is odd and the sets $\partial\mathcal A_*$ and $\mathcal A_k,$ $k\geq 1$ are symmetric with respect to the origin.
\\
For a closed symmetric subset $B\subset\partial\mathcal A_*$ we denote by $\gamma (B)$ the usual Krasnoselsii genus and we recall some of the properties we will need:
\begin{lemma} Let $A, B\subset\partial\mathcal A_*$ be closed and symmetric.
\begin{itemize}
\item[(i)] If $A\subset B,$ then $\gamma(A)\leq \gamma(B).$
\item[(ii)] If $h:A\rightarrow\partial\mathcal A_* $ is continuous and odd, then $\gamma(A)\leq\gamma(\overline{h(A)}).$ 
\item[(iii)] If $S$ is a bounded symmetric neighborhood of the origin in a $k$-dimensional normed vector space and $v:\partial S\rightarrow \partial \mathcal A_*$ is continuous and odd,  then $\gamma(v(\partial S))\geq k. $
\end{itemize}
\end{lemma}
$\ $\\
Let $$\mathcal O:=\mathcal A_*\cap W_k,$$
where $W_k$ is the $k$-dimensional subspace of $X$ defined at Section \ref{SPazioFinitoDim} (Lemma \ref{lemmaWminore}% and Lemma \ref{lemmaWuguale} resp.
).
$\mathcal O$ is a symmetric, bounded (from \eqref{unboundedBelow2}% resp. \eqref{unboundedBelow1}
) open neighborhood of $0$ on $W_k.$ 

\begin{lemma}\label{lemmaAk}
$\partial \mathcal O\subset \mathcal A_k$
and $\gamma (\partial \mathcal O)=\gamma(\mathcal A_k)=k$
\end{lemma}

\begin{proof}
The inclusion $\partial \mathcal O\subset \mathcal A_k$  is a consequence of the definition of $W_k.$ Let us compute the genus.
From the property (iii) of the genus it follows immediately that $\gamma (\partial \mathcal O)\geq k.$ Moreover, adapting the arguments in the proof of \cite[Lemma 3.3]{WeiWeth}, one can prove that
$\gamma (\mathcal A_k)\leq k.$ The conclusion comes from the monotonicity property (i) of the genus.
\end{proof}
We define also the closed subsets of $\partial \mathcal A_*$
$$\mathcal C^t_{k-1}:=\{u\in\partial\mathcal A_*:\varphi^t(u)\in \mathcal A_{k-1}\}\quad\mbox{ for }t>0.$$ 
\begin{lemma}\label{lemmaC}
 $\mathcal A_{k-1}\subset \mathcal C^t_{k-1}$ and
$\gamma(\mathcal C^t_{k-1})=\gamma (\mathcal A_{k-1})= k-1$ for every $t>0.$
\end{lemma}
\begin{proof} The proof  is trivial once we know that $\gamma (\mathcal C^t_{k-1})\leq k-1.$ 

This is a consequence of the property (ii) of the genus, indeed $\gamma (\mathcal C^t_{k-1})\leq \gamma (\overline{\varphi^t(\mathcal C^t_{k-1})})$ since the map $\varphi^t:\mathcal C^t_{k-1}\rightarrow \partial \mathcal A_*$ is continuous and odd, and $\gamma (\overline{\varphi^t(\mathcal C^t_{k-1})})\leq k-1$ because of the inclusion $\overline{\varphi^t(\mathcal C^t_{k-1})}\subset \mathcal A_{k-1}.$
\end{proof}
\begin{proposition}\label{propositionSolution}
There exists $\bar u\in \partial \mathcal O\setminus \mathcal A_{k-1}$ such that $\emptyset\neq\omega(\bar u)\subset\mathcal A_{k}\setminus\mathcal A_{k-1}.$
\end{proposition}
\begin{proof}
From Lemma \ref{lemmaAk} and Lemma \ref{lemmaC} $\emptyset\neq\partial \mathcal O\setminus \mathcal C^t_{k-1}\subset \mathcal A_k\setminus\mathcal A_{k-1}$ for every $t>0.$ 
\\
In particular for any positive integer $n$ there exists $u_n\in \partial \mathcal O\setminus \mathcal C^n_{k-1}$ and,
since $\partial\mathcal O$ is compact, we may pass to a subsequence such that $u_n\rightarrow \bar u\in\partial\mathcal O$ as $n\rightarrow \infty.$
\\
Obviously $\omega(\bar u)\subset \mathcal A_k, $ as in \cite{WeiWeth} we now prove that $\omega(\bar u)\subset \mathcal A_k\setminus\mathcal A_{k-1}.$

On this scope we define the sets
$$Y_k:=\{u\in Y:i(u)\leq k-1\}$$
(see  Section \ref{sectionParabolic} for the definition of the space $Y$).  By construction (using the continuity property in Corollary \ref{continuitaDatoIn})
one has that
$\varphi^t(\bar u)\not\in Int_{Y}(Y_{k-1})$ for every $t>0,$ which implies that $\omega(\bar u)\cap Int_{Y}(Y_{k-1})=\emptyset.$ 

On the other hand if we assume  by contradiction that  $\omega(\bar u)\cap \mathcal A_{k-1}\neq \emptyset$ than, since 
$\omega(\bar u)$ consists of radial solutions of \eqref{problemaDiDirichletNellaPalla},   one can asily show (with arguments similar to the ones in \cite[Lemma 3.1]{WeiWeth}) that $\omega(\bar u)\cap Int_{Y}(Y_{k-1})\neq\emptyset,$ reaching a contradiction.  
\end{proof}

$\;$\\
\\
\\
The proof of Theorem \ref{teoremaEsistenzaNellaPalla} follows from Proposition \ref{propositionSolution} taking any $u\in w(\bar u).$
$u$ is a radial solution for the elliptic problem \eqref{problemaDiDirichletNellaPalla} with exactly $k-1$ changes of sign.
Moreover, since the energy is non-increasing along trajectories and using the energy estimate in $W_k,$ (see \eqref{stimaE2}%resp. \eqref{stimaE1}
) it satisfies
 $$E(u)\leq E(\bar u)\leq C_k.$$
%%
%%
%%
%% when $q>3.$
%%
%%
%%
%
%\begin{remark}
%when $q=3$....the bound is not uniform in the radius..........
%\end{remark}
$$$$

As a byproduct we obtain also the proof of Theorem \ref{teoParabolico} related to the existence of global solutions of the parabolic problem \eqref{IBVP} with the same fixed number of nodal regions along the flow.

Indeed $u_k(t):=\varphi^t(\bar u)$ is a global solution of the parabolic problem and, from the positive invariance of the sets $\mathcal A_k,$ it follows that $u_k(t)\in \mathcal A_k\setminus \mathcal A_{k-1}$ for every $t\geq 0.$

$$$$

\section{Proof of Theorem \ref{teoremaEsistenzaInRN}}\label{sectionTeoInR3}
%{\textcolor{red}{intro su quello che facciamo in questa sezione}}
For fixed $k\geq 2,$
let $R_n\geq 1, n\in \mathbb N$ such that $R_n\rightarrow +\infty$ as $n\rightarrow +\infty,$ let $u_n\in H^1_{0}(\B_n)$ be a radial weak solution of the Dirichlet problem in the ball $\mathbb B_n=\mathbb B_{R_n}$
\begin{equation}\label{DirichletBn}\left\{
\begin{array}{lr}
-\Delta u_n +u_n+ u_n\int \frac{u_n^2(y)}{|x-y|}dy-|u_n|^{q-1}u_n=0\quad\mbox{ in }\mathbb B_n\\
u_n=0\quad\mbox{ on }\partial\mathbb B_n
\end{array}
\right.\end{equation}
with precisely $(k-1)$ changes of sign in the radial variable and which satisfies the energy uniform bound 
$E(u_n)\leq M_k$ (from Theorem \ref{teoremaEsistenzaNellaPalla}).\\\\

\begin{proposition}[Uniform $H^1$-bound]\label{UniformHBound}There exists $D_k>0$ (independent of $n$) such that
$$\|u_n\|\leq D_k\quad\forall n\in\mathbb N.$$
\end{proposition}

\begin{proof}
Since
$u_n$ is a solution it satisfies
$$E'(u_n)(u_n)=0,$$ hence, since $q\geq 3$
%
%  Metto inssieme 1) e 2):
%  
%  $$\left\{\begin{array}{lr}
%	M\geq E(u)=\frac{1}{2}\|u\|^2+\frac{1}{4}\int\int\frac{u^2(x)u^2(y)}{|x-y|}dxdy-\frac{1}{q+1}\|u\|_{q+1}^{q+1}\\
%	0=E'(u)(u)=\|u\|^2+\int\int\frac{u^2(x)u^2(y)}{|x-y|}dxdy-\|u\|_{q+1}^{q+1}
%\end{array}\right.$$
$$M_k\geq E(u_n)=E(u_n)-\frac{1}{4}E'(u_n)(u_n)=\frac{1}{4}\|u_n\|^2+\frac{q-3}{4(q+1)}\|u_n\|_{q+1}^{q+1}\geq \frac{1}{4}\|u_n\|^2,$$
namely
$$\|u_n\|\leq 2\sqrt{M_k}.$$
\end{proof}

\begin{lemma}[Regularity of $u_n$ and uniform $C^{2,\alpha}_{loc}$-bound]\label{Regularity}\label{localUniformBound}
$\ $
$u_n\in C^{2,\alpha}(\bar \B_n)$ and it is a classical solution of the Dirichlet problem \eqref{DirichletBn}. In particular  
for any $R>0$  there exist $n_R\in \mathbb N$ and $C_R>0$ such that 
$$ u_n\in C^{2,\alpha}(\bar \B_R)\ \mbox{ and }\  \|u_n\|_{C^{2,\alpha}(\bar \B_R)}\leq C_R\ \mbox{ for all }n\geq n_R . $$
(The constant $C_R$ is independent of $n$ but depends on $R, q, \alpha, D_k$).
\end{lemma}

\begin{proof}
%We denote by $C$ any constant which doesn't depend on $n\in\mathbb N$ but which may eventually depend on $q, \alpha, R.$\\
Let $$\phi_n(x):=\phi_{u_n}(x)=\int \frac{u_n^2(y)}{|x-y|}dy.$$
{\bf STEP 1} {\it $\phi_n\in C^{0,\alpha}(\mathbb R^3).$ Moreover for any $R>0$ there exists $C=C(\alpha, R, D_k)>0$ such that $$\|\phi_n\|_{C^{0,\alpha}(\bar \B_R)}\leq C\ \forall n\in \mathbb N.$$
}
\\
\\
Since $u_n\in H^1(\mathbb R^3)$, it can be proved that $\phi_n\in D^{1,2}(\mathbb R^3)$ and that the following  bound holds \begin{equation}\label{BoundPhiD}\|\phi_n\|_{D^{1,2}}\leq C\|u_n\|^2\end{equation}
where $C>0$ is independent of $n\in\mathbb N$ (see Lemma \eqref{preliminariSuPhi}).\\
\\\\
Moreover, by the Sobolev embedding $H^1(\mathbb R^3)\hookrightarrow L^6(\mathbb R^3)$ we deduce that $u_n^2\in L^3(\mathbb R^3)$, hence (using for instance \cite[Theorem 9.9]{GilbargTrudinger} in any domain $\Omega$ which contains $\bar \B_n$)
 $\phi_n\in W^{2,3}_{loc}(\mathbb R^3)$ and $-\Delta \phi_n=u_n^2$ a.e. in $\mathbb R^3.$
\\
\\
Therefore from \cite[Theorem 9.11]{GilbargTrudinger}, Sobolev embeddings and \eqref{BoundPhiD} we obtain
\begin{eqnarray*}\|\phi_n\|_{W^{2,3}(\B_R)} &\leq& C(R)\left(\|\phi_n\|_{L^3(\B_{2R})}+\|u_n^2\|_{L^3(\B_{2R})}\right)\\
&& \leq C(R)\left(\|\phi_n\|_{D^{1,2}}+\|u_n\|^2\right)\\
&& \leq C(R)(1+C)\|u_n\|^2.
\end{eqnarray*}
The conclusion follows from the Sobolev embedding $W^{2,3}(\B_R)\hookrightarrow C^{0,\alpha}(\bar \B_R)$ and Proposition \ref{UniformHBound}.\\
\\
{\bf STEP 2} {\it $u_n\in C^{0,\alpha}(\mathbb R^3).$ In particular for any $R>0$ there exists %%%%%%%%%%%%%%$n_R\in\mathbb N,$ 
$C=C(q, \alpha, R, D_k)>0$ such that 
$$u_n\in C^{0,\alpha}(\bar \B_R)\ \ \mbox{  and } \ \ \|u_n\|_{C^{0,\alpha}(\bar \B_R)}\leq C \quad \forall n %%%%%%%%%%%%%\geq n_R
$$
(here $u_n$ stays for its trivial extension).
}
\\
\\
$u_n$ is a weak solution of the Dirichlet problem
\begin{equation}\label{DirichletConF}\left\{
\begin{array}{lr}
-\Delta u_n =f_n\quad\mbox{ in }\mathbb \B_n\\
u_n=0\quad\mbox{ on }\partial\mathbb \B_n
\end{array}
\right.\end{equation}
where $f_n:=|u_n|^{q-1}u_n-(1+\phi_n)u_n\in L^{2^*/q}(\B_n)$ ($\phi_n\in C^0(\bar \B_n)$ by Step 1).
By $L^p$-regularity it follows that $u_n\in W^{2,2^*/q}(\B_n)$ and so by a classical bootstrap argument $u_n\in W^{2,p}(\B_n)$ for a certain $p>\frac{N}{2};$  by Sobolev embeddings we conclude that $u_n\in C^{0,\alpha}(\bar \B_n).$ Substituting  $u_n$ by its trivial extension (remember that $u_n=0$ on $\partial \B_n$) one obtains that $u_n\in C^{0,\alpha}(\mathbb R^3).$
\\
\\
We prove now the uniform $C^{0,\alpha}_{loc}$-estimate.
We denote by $C$ any constant which doesn't depend on $n\in\mathbb N$ but which may eventually depend on $q, \alpha, R, D_k$ and which may vary from line to line.\\
Let us observe that for each fixed $R>0,$ there exists $n_R\in\mathbb N$ such that $ \B_{2R}\subseteq \B_n$  $\forall n\geq n_R.$
Therefore from Sobolev embeddings ($p>\frac{N}{2}$) we have for $n\geq n_R$
\begin{equation}\label{sobolevC}\|u_n\|_{C^{0,\alpha}(\bar \B_R)}\leq C\|u_n\|_{W^{2,p}(\B_R)};
\end{equation}
moreover from $L^p$-estimates (cfr. \cite[Theorem 9.11]{GilbargTrudinger}) and Step 1 
\begin{eqnarray}\label{srimaregprec}  \|u_n\|_{W^{2,p}(\B_R)}
&\leq & C\left(\|u_n\|_{L^{p}(\B_{2R})}+\|f_n\|_{L^{p}(\B_{2R})}\right)\nonumber\\
&&\leq  C\left(\|u_n\|_{L^{p}(\B_{2R})}+\|u_n\|^q_{L^{qp}(\B_{2R})}+(1+\|\phi_n\|_{C^{0,\alpha}(\bar \B_{2R})})\|u_n\|_{L^{p}(\B_{2R})}\right)\nonumber\\
&&\leq C\left(\|u_n\|^q_{L^{qp}(\B_{2R})}+\|u_n\|_{L^{pq}(\B_{2R})}\right)\quad(pq>p)
\end{eqnarray}
Remember now that $p\geq \frac{2^*}{q}$ is obtained after a finite number $m\geq 0$ of iterations (bootstrap procedure) starting from $\frac{2^*}{q}.$ If $p=\frac{2^*}{q}$ we can conclude directly  from \eqref{sobolevC} and \eqref{srimaregprec}, using Sobolev embeddings and Proposition \ref{UniformHBound}, indeed 
$$\|u_n\|_{C^{0,\alpha}(\bar \B_R)}\leq C\left(\|u_n\|^q+\|u_n\|\right)\leq C\left(D_k^q+D_k\right)=C.$$
Otherwise, if $p>\frac{2^*}{q}$ then it is easy to verify that we reduce to the previous case iterating $m$ times  the estimate \eqref{srimaregprec} together with the Sobolev embeddings
$W^{2,\frac{Ns}{N+2s}}\hookrightarrow L^s,$ for opportune $s>0.$ 
We remark that at each step the constant involved is independent of $n.$\\

The previous argument gives the local estimate only definitely for $n\geq n_R;$
To extend the result also to the first $n_R-1$ elements it is enough to substitute $u_n$ with its trivial extension  for each $n\leq n_R-1$ (in this case $\B_n\subset \B_{2R}$). 
%Indeed it is easy to verify that the trivial extension of $u_n$ belongs to $C^{0,\alpha}(B_R),$ since $u_n=0$ on $\partial B_n.$ 
\\
\\
%$$\|u_n\|_{L^{pq}}\stackrel{(Sobolev)}{\leq} C\|u_n\|_{W^{2,\frac{Npq}{N+2pq}}}\stackrel{\eqref{srimaregprec}}{\leq}\|\|.....\leq %C\|u_n\|_{W^{2,2^*}(B_R)} \stackrel{\eqref{srimaregprec}}{\leq}$$
{\bf STEP 3} {\it Conclusion.
}
\\
\\
$u_n$ is a weak solution of the Dirichlet problem \eqref{DirichletConF}, where $f_n\in C^{0,\alpha}(\bar \B_n)$ (from Step 1 and Step 2).
By elliptic regularity we conclude that $u_n\in C^{2,\alpha}(\bar \B_n)$ and it is a classical solution of the Dirichlet problem.

To prove the uniform $C^{2,\alpha}_{loc}$-estimate we fix any $R>0$ and let $n_R\in\mathbb N$ such that $ \B_{2R}\subseteq \B_n$  $\forall n\geq n_R.$ Hence for $n\geq n_R$
$$\|u_n\|_{C^{2,\alpha}(\bar \B_R)}
\leq  C\left(\|u_n\|_{C^{0,\alpha}(\bar \B_{2R})}+\|f_n\|_{C^{0,\alpha}(\bar \B_{2R})}\right)\le C,$$
where we have used the Schauder estimates (see for instance \cite[Theorem 4.6]{GilbargTrudinger}) and the uniform local bounds from Step 1 and Step 2.
\end{proof}

\begin{proposition}\label{convergenzaUN}
There exists $u\in C^2(\mathbb R^3)\cap H^1(\mathbb R^3),$ radial such that
$$u_{n_k}\rightarrow u \mbox{ in }C^2_{loc}(\mathbb R^3).$$
Moreover $u$ is a solution of
\begin{equation}\label{SPSEquation}
-\Delta u +u+u\int \frac{u^2(y)}{|x-y|}dy-|u|^{q-1}u=0\quad\mbox{ in }\mathbb R^3
\end{equation}
\end{proposition}

\begin{proof}
Since $u_n$ is bounded in $H^1(\mathbb R^3)$ (Proposition \ref{UniformHBound}), one can extract a subsequence of $u_n,$ again denoted by $u_n,$ such that $u_n$ converges weakly in $H^1(\mathbb R^3)$ and almost everywhere in $\mathbb R^3$ to a function $u.$
Observe that $u\in H^1(\mathbb R^3)$ is spherically symmetric.
Moreover, since for any $R>0$ the sequence $u_n$ is definitely bounded in $C^{2,\alpha}(\B_R)$ (Lemma \ref{localUniformBound}), by Arzela's theorem and a standard diagonal process one can also prove that $u\in C^2(\mathbb R^3)$ and that $u_n$ converges to $u$ in $C^2_{loc}(\mathbb R^3).$

In order to prove that $u$ satisfies equation \eqref{SPSEquation} it is enough to pass to the limit for a.e. $x\in\mathbb R^3$ into the equation  pointwise satisfied by $u_n$ definitely (by Lemma \ref{Regularity}) 
$$
-\Delta u_n(x) +u_n(x)+u_n(x)\phi_n(x)-|u_n(x)|^{q-1}u_n(x)=0
$$
To this scope let's observe that from the compactenss result iv) in Lemma \ref{preliminariSuPhi} one can extract a subsequence of $\phi_n,$ again denoted by $\phi_n,$ such that $\phi_n$ converges (strongly in $D^{1,2}(\mathbb R^3)$ and hence, by Sobolev embedding $D^{1,2}(\mathbb R^3)\hookrightarrow L^6(\mathbb R^3)$) almost everywhere in $\mathbb R^3$ to the function $\phi_u=\int \frac{u^2(y)}{|x-y|}dy.$ 
\end{proof}

\begin{lemma}\label{lemmaMaggioreRegolaritaUn}
	 $\phi_n\in C^{2,\alpha}(\mathbb R^3)$ and it is a classical solution of the equation $-\Delta\phi_n=u_n^2$ in $\mathbb R^3.$
	Moreover for any $R>0$  there exists  $D_R>0$ such that 
$$ \phi_n\in C^{2,\alpha}(\bar \B_R)\ \mbox{ and }\  \|\phi_n\|_{C^{2,\alpha}(\bar \B_R)}\leq D_R\ \mbox{ for all }n\in\mathbb N. $$
(The constant $D_R$ is independent of $n$ but depends on $R, q, \alpha, D_k$).
\end{lemma}
\begin{proof} %  for any $R>0$  $\|u_n\|_{C^{0,\alpha}(\bar B_R)}\leq\|u_n\|_{C^{0,\alpha}(\bar B_n)}.$  
%
%
%{\it $u_n\in C^{0,\alpha}(\mathbb R^3)$, $\phi_n\in C^2_{loc}(\mathbb R^3)$ and satisfies the equation $-\Delta\phi_n=u_n^2$ in $\mathbb R^3.$
%}
%
One can easily prove that $u_n^2\in C^{0,\alpha}(\mathbb R^3)$ and it is bounded (see Lemma \ref{Regularity}-Step 2). Therefore it follows that $\phi_n\in C^2(\mathbb R^3)$ and satisfies the equation $-\Delta\phi_n=u_n^2$ in $\mathbb R^3$ (using for instance the known regularity result about newtonian potentials in \cite[Lemma 4.2]{GilbargTrudinger} in any domain $\Omega$ which contains $\bar \B_n$).

The conclusion follows from interior H\"older estimates for solutions of the Poisson's equation (for instance \cite[Theorem 4.6]{GilbargTrudinger}), namely $\phi_n\in C^{2,\alpha}(\mathbb R^3)$ and, fixed any $R>0$ one has
$$\|\phi_n\|_{C^{2,\alpha}(\bar \B_R)}\leq C \left(\|\phi_n\|_{C^0(\bar \B_{2R})}+\|u_n^2\|_{C^{0,\alpha}(\bar \B_{2R})}\right)\leq C\quad \forall n$$
where the last inequality follows from Step 1 and Step 2 in Lemma \ref{Regularity}.
\end{proof}

\begin{proposition}\label{phiVerificaLaPoisson}
	 $\phi_u:=\int\frac{u^2(y)}{|x-y|}dy\in C^{2}(\mathbb R^3)\cap D^{1,2}(\mathbb R^3)$ and it is a classical radial solution of the equation $-\Delta\phi_u=u^2$ in $\mathbb R^3.$
%	\textcolor{blue}{ (Quindi ho trovato una soluzione classica del sistema.....)}
\end{proposition}
\begin{proof}
From Lemma \ref{lemmaMaggioreRegolaritaUn}, using Arzela's theorem and a standard diagonal process, one can extract a subsequence of $\phi_n$, again denoted by $\phi_n,$ which converges in $C^2_{loc}(\mathbb R^3$) to a function $w\in C^2(\mathbb R^3).$ Moreover (Proposition \ref{convergenzaUN})  one can extract a subsequence of $u_n,$ again denoted by $u_n,$ which converges pointwise in $\mathbb R^3$ to the function $u.$
Therefore passing to the limit for a.e. $x$ in $\mathbb R^3$ into the equation
$-\Delta\phi_n(x)=u^2_n(x),$ one can prove that
$w$ is a classical solution of the equation $-\Delta w= u^2$ in $\mathbb R^3.$ 

Last it is clear that $w$ coincides with $\phi_u$ since we already know that $\phi_n\rightarrow \phi_u$ a.e. in $\mathbb R^3$ (see the proof of Proposition \ref{convergenzaUN}).

To conclude we observe that $\phi_u\in D^{1,2}(\mathbb R^3)$ and it's radial because $u\in H^1(\mathbb R^3)$ and it's radial (see Lemma \ref{preliminariSuPhi}).
\end{proof}

$$$$
Propositions \ref{convergenzaUN} and \ref{phiVerificaLaPoisson} yield the existence of a radial solution $(u,\phi_u)$ for system \eqref{sistemaSPSInRN}.

Next we prove that $u$ is nontrivial and has exacly $(k-1)$ changes of sign in the radial variable. 
$$$$

\begin{lemma}\label{LemmaMinMaxLocali}
Let $\bar r>0$ be a positive local maximum or a negative local minimum point for $r\mapsto u(r)$ (resp. $r\mapsto u_n(r)$).
Then $$|u(\bar r)|\ \mbox{(resp. $\ |u_n(\bar r)|$) }\ \geq 1$$
\end{lemma}

\begin{proof}
Let $w=u$ (resp. $u_n$), so writing the equation \eqref{DirichletBn} (resp. \eqref{SPSEquation}) in polar coordinates:
$$-  w''(\bar r)-\frac{2}{\bar r} w'(\bar r)=w(\bar r)\left[| w(\bar r)|^{q-1}-(1+\phi_{w}(\bar r))\right].$$
If $\bar r>0$ is a local maximum (resp. a local minimum) point for $w$ then
$$w(\bar r)\left[| w(\bar r)|^{q-1}-(1+\phi_{ w}(\bar r))\right]\geq 0\quad \mbox{(resp. $\leq0$)}$$
hence, since
$w(\bar r)>0$ (resp. $w(\bar r)<0$) 
$$| w(\bar r)|^{q-1}\geq(1+\phi_{ w}(\bar r))\geq 1$$
namely the thesis.
\end{proof}

\begin{lemma}\label{LemmaNonScappanoZeriAdInfinito} There exists $R>0$ such that
any positive local maximum or negative local minimum point $r\geq 0$ for the function $r\mapsto u(r)$ or $r\mapsto u_n(r),$ satifies
$$r\leq R$$
($R$ doesn't depend on $u$ nor $n$ but may depend on $k$).
\end{lemma}

\begin{proof} 
By Strauss Lemma\footnote{Strauss Lemma: let $N \geq 2$; every radial function $u\in H^1(R^N)$ is almost
everywhere equal to a function $U(x)$, continuous for $x \neq 0$ and such that
$$|U(x)| \leq  C_N |x|^{(1-N)/2} \|u\|_{H^1(\mathbb R^N)}\ \mbox{ for }\ |x|\geq R_N$$
where $C_N$ and $R_N$ depend only on the dimension $N$} (see \cite{BerLionsI}) and the uniform bound on the $H^1$ norms there exists $R>0$ such that 
\begin{equation}\label{assurdo}|u(x)|,| u_n(x)|<\frac{1}{2},\quad\mbox{ for any } \ x\in\mathbb R^3\setminus \B_R,\quad\mbox{ for any } n\end{equation}
We prove the result for $u_n.$ The proof for $u$ can be done in a similar way.\\
\\
For small $n$ such that $R_n\leq R$ the result is trivial (since $u_n\equiv 0$ in $\mathbb R^3\setminus \B_R$).
Hence let's consider $n$ such that $R_n> R$ and
by contradiction let $r_n>R$ be a maximum (resp. a minimum) for $u_n$ with
$u_n(r_n)>0$ (resp. $u_n(r_n)<0$).
Hence by Lemma \ref{LemmaMinMaxLocali}
$$|u_n(r_n)|\geq 1$$ which contradicts \eqref{assurdo}.
\end{proof}

\begin{lemma}[Properties of $u$]\label{proprietaU} Let $R>0$ be as in Lemma \ref{LemmaNonScappanoZeriAdInfinito}. 
\begin{enumerate}
\item[i)] $u\neq 0$;
\item[ii)]  $u$ changes sign;
\item[iii)]  let $\bar r$ such that $u(\bar r)=0$, then $\frac{\partial u}{\partial r}(\bar r)\neq 0;$
\item[iv)] let $\bar r$ such that $u(\bar r)=0$, then $\bar r\in (0,R)$ and it is isolated.
\item[v)] in every subinterval where $r\mapsto u(r)$ changes sign precisely once, $r\mapsto u_n(r)$ also changes sign precisely once for large $n$.
\end{enumerate}
\end{lemma}

\begin{proof}
To prove i) we assume by contradiction that $u\equiv 0.$ 
Then, since $u_n$ converges to $u$ in $C^2_{loc}(\mathbb R^3)$ (Proposition \ref{convergenzaUN}), in particular it follows that 
\begin{equation}\label{assurdoMAx}\max_{x\in\bar{\B_R}}|u_n(x)-u(x)|=\max_{x\in\bar{\B_R}}|u_n(x)|\rightarrow_n 0.\end{equation}
But we know that for each $n\in \mathbb N$ the function $r\mapsto u_n(r)$ changes sign exacly $(k-1)$ times and it is regular, hence it has at least $(k-1)$ positive maximum/negative minimum points and moreover (Lemma \ref{LemmaNonScappanoZeriAdInfinito}) these points are all inside $\B_R.$   
Therefore, from Lemma \ref{LemmaMinMaxLocali} it follows that
% each of these points $\bar r$ satisfy $|u_n(\bar r)|\geq 1.$
$$\max_{x\in\bar{\B_R}}|u_n(x)|\geq 1\quad\forall n\in\mathbb N,$$
which contradicts \eqref{assurdoMAx}.
\\

To prove ii) we assume by contradiction that $u\geq 0.$ Then $u>0$ by the strong maximum principle and hence 
there exists $C>0$ such that $u\geq C$ in $\bar \B_R.$
Since $u_n$ converges to $u$ in $C^2_{loc}(\mathbb R^3)$ (Proposition \ref{convergenzaUN}), 
one has that $u_n\geq \frac{C}{2}>0$ in $\bar \B_R$ definitely.
Which is absurd since $u_n$ changes sign is $\B_R.$
\\

iii) is a direct consequence of the Hopf's boundary Lemma
%\\ \textcolor{blue}{CANCELLARE?: Indeed let $\bar r$ such that $u(\bar r)=0$ then 
% $\bar r$ belongs to the boundary of a connected nodal region $N^+$ of $u$ such that  $u>0$ in $N^+.$
% Since $u$ is radial $N^+$ is a ball or a ring or a ball without center (all domains which satisfy the interior ball condition).}\\
while iv) follows immediately from iii) and Lemma \ref{LemmaNonScappanoZeriAdInfinito}. 
  \\
 
Last we prove v).
Let $(a,b)$ be an interval where $r\mapsto u(r)$ changes sign precisely once and let $\bar r\in (a,b)$ be the unique point in $(a,b)$ such that $u(\bar r)=0.$ 

We want to prove that $u_n$ changes sign precisely once in $(a,b)$ for large $n.$ 

By Lemma \ref{LemmaNonScappanoZeriAdInfinito} we may restrict w.l.o.g  to the case $(a,b)\subseteq (0,R),$ in particular $(a,b)$ is a bounded interval.
 
Since $u_n$ converges uniformly to $u$ on compact intervals (Proposition \ref{convergenzaUN}), one easily deduces that for $n$ large $u_n$ changes sign at least once in $(a,b).$ 
 
On the other hand from  iii) we know that $\frac{\partial u}{\partial r}(\bar r)\neq 0,$ w.l.o.g. we may assume for instance that $\frac{\partial u}{\partial r}(\bar r)> 0.$ By continuity $\frac{\partial u}{\partial r}(r)\geq \alpha>0$
 in a neighborhood $O_{\bar r} \subseteq (a,b)$ of $\bar r,$
 and from the $C^2_{loc}$-convergence, it follows that for large $n$
$$\frac{\partial u_n}{\partial r}(r)\geq \frac{\alpha}{2}>0\ \ \forall r\in O_{\bar r},$$
namely for large $n$ the function $r\mapsto u_n(r)$ is strictly monotone in $O_{\bar r}.$ 

Moreover by assumption, $|u(r)|\geq\beta>0$ for $r\in(a,b)\setminus O_{\bar r},$ therefore from the $C^2_{loc}$-convergence also 
$|u_n(r)|\geq\frac{\beta}{2}>0$ for $r\in(a,b)\setminus O_{\bar r},$ for large $n.$

Hence we can conclude that for large $n$ the function  $r\mapsto u_n(r)$ changes sign inside $O_{\bar r}$ and exactly once because of the strict monotonicity.
\end{proof}

\begin{proposition}\label{uCambiaSegno} $u$ changes sign precisely $(k-1)$ times in the radial variable.
\end{proposition}
\begin{proof} From Lemma \ref{proprietaU} we know that there exists an integer $m\geq 1$ such that the function $r\mapsto u(r)$ changes sign $m$ times. 
In particular there exist $m$ isolated points $0<r_1<r_2<..< r_m<R$ such that $u(r_i)=0$ for any $i=1,..,m$ and  $u(r)\neq 0$ for $r\neq r_i$ $i=1,..,m.$

Let us define the partition $0=x_0<x_1<...<x_{m-1}<x_m=R$ of $(0,R)$ where $x_i=(r_i+r_{i+1})/2,$ $i=1,.., m-1;$ 
and let us consider the subintervals $I_k:=(x_k,x_{k+1})$ for any $k=0,.., m-1.$

By construction the function $r\mapsto u(r)$ changes sign exactly once in each $I_k$ %(in the point $r_{k+1}$)
and so,
by point v) in Lemma \ref{proprietaU}, it follows that for $n$ large the function $r\mapsto u_n(r)$ changes sign exactly once in each $I_k.$

Therefore in the interval $(0, R)$ the function $r\mapsto u_n(r)$ changes sign exactly $m$ times for $n$ large and by Lemma \ref{LemmaNonScappanoZeriAdInfinito} it changes sign exacly $m$ times at all. As a consequence $m=k-1.$ 
\end{proof}

$\ $\\\\\\\\\\Theorem \ref{teoremaEsistenzaInRN} is a direct consequence of Propositions \ref{convergenzaUN}, \ref{phiVerificaLaPoisson} and \ref{uCambiaSegno}.
$$$$

\section*{Acknowledgments} The author would like to express his sincere gratitude to professor Tobias Weth for bringing to his attention paper \cite{WeiWeth} as well as for his many helpful advices and fruitful discussions during his stay in the Goethe Universit\"at of Frankfurt-am-Main.

\end{document}